\documentclass[10pt]{article}

\usepackage[latin1]{inputenc}
\usepackage{mathrsfs}
\usepackage{amsmath,amssymb}
\usepackage{verbatim}	
\usepackage{amsthm}		
\usepackage{array}		
\usepackage{subfigure}
\usepackage{enumerate}

\usepackage{graphicx}

\newtheorem{theorem}{Theorem}
\newtheorem{lemma}[theorem]{Lemma}
\newtheorem{corollary}[theorem]{Corollary}
\newtheorem{proposition}[theorem]{Proposition}

\def\P{\mathscr{P}}

\def\C{\mathcal{C}}

\begin{document}

\begin{center}
  \Large\bf{All-path convexity: Combinatorial and complexity aspects}
\end{center}
\smallskip
\begin{center}
  F\'abio Protti\footnote{Partially supported by CNPq and FAPERJ, Brazilian Research Agencies}\\
  Instituto de Computação - Universidade Federal Fluminense -  Brazil\\
  E-mail: {\tt fabio@ic.uff.br}\\
  \hspace*{1cm}\\
  Jo\~ao V. C. Thompson\\
  Centro Federal de Educação Tecnológica Celso Suckow da Fonseca\\
  CEFET/RJ  - Campus Petrópolis -  Brazil\\
  E-mail: {\tt joao.thompson@cefet-rj.br}
\end{center}

\smallskip

\noindent {\bf Abstract.} Let $\P$ be any collection of paths of a graph $G=(V,E)$. For $S\subseteq V$, define $I(S)=S\cup\{v\mid v \ \mbox{lies in a path of} \ \P \ \mbox{with endpoints in} \ S\}$. Let $\C$ be the collection of fixed points of the function $I$, that is, $\C=\{S\subseteq V\mid I(S)=S\}$. It is well known that $(V,\C)$ is a finite convexity space, where the members of $\C$ are precisely the convex sets. If $\P$ is taken as the collection of all the paths of $G$, then $(V,\C)$ is the {\em all-path convexity} with respect to graph $G$. In this work we study how important parameters and problems in graph convexity are solved for the all-path convexity.

\smallskip


\section{Introduction}\label{sec:intro}

Let $G=(V,E)$ be a simple, finite, nonempty, and connected graph, and let $\C$ be a collection of subsets of $V$. We say that $(V,\C)$ is a (finite) {\em graph convexity space} if: (a) $\emptyset\in\C$; (b) $V\in\C$; (c) $\C$ is closed under intersections. In many studies, the collection $\C$ is determined as follows. Let $\P$ be any collection of paths of a graph $G$, and, for $S\subseteq V$, define
$$I(S)=S\cup\{v\mid v \ \mbox{lies in a path of} \ \P \ \mbox{with endpoints in} \ S\}.$$ Define $\C\subseteq 2^V$ as the collection of fixed points of the function $I$, that is, $\C=\{S\subseteq V\mid I(S)=S\}$. Then $(V,\C)$ is easily seen to be a graph convexity space, generally called a {\em ``path convexity''}. In particular, if $\P$ contains precisely all the shortest paths of $G$ then the corresponding convexity space is the well-known {\em geodesic convexity} with respect to $G$ \cite{P113,HN81,P148}; if $\P$ is the collection of induced paths of $G$ then the corresponding convexity space is the {\em monophonic convexity} with respect to $G$ \cite{D10,D88}; and there are in the literature many other examples described below, where in each case we indicate which collection $\P$ of paths of $G$ is considered:
\begin{itemize}
\setlength\itemsep{0em}
\item[--] $g^3$-convexity \cite{P150}: shortest paths of length at least three;
\item[--] $m^3$-convexity \cite{P28,P88}: induced paths of length at least three;
\item[--] $g_k$-convexity \cite{P101}: shortest paths of length at most $k$;
\item[--] $P_3$-convexity \cite{P37,P85,P153}: paths of length two;
\item[--] $P_3^*$-convexity \cite{A13}: induced paths of length two;
\item[--] triangle-path convexity \cite{P40,P45}: paths allowing only triangular chords;
\item[--] total convexity \cite{P76}: paths allowing only non-triangular chords;
\item[--] detour convexity \cite{P46,P70,P68}: longest paths.
\end{itemize}

In this work, we study the {\em all-path convexity}, which is associated with the collection $\P$ of all the paths of $G$ \cite{P42,P107,P165}. Our study is concentrated in solving the most important problems in graph convexity for the specific case of the all-path convexity, including the determination of some well-known graph convexity parameters, such as the {\em convexity number}, the {\em interval number}, and the {\em hull number} of $G$. The remainder of this work is organized as follows. In Section 2 we provide some necessary background. In Section 3 we present the main results of this work. Finally, Section 4 contains our concluding remarks.

\section{Preliminaries}

In this work, $G=(V,E)$ is always a finite, simple, nonempty, and connected graph with $n$ vertices and $m$ edges. Let $\P$ be a collection of paths of $G$, and let $I_{\P}:2^{V}\rightarrow 2^{V}$ be a function ({\em interval function}) associated with $\P$ in the following way:
$$I_{\P}(S) = S\cup\{v\mid v \ \mbox{lies in a path} \ P\in\P \ \mbox{with endpoints in} \ S\}.$$



Define $\C_{\P}$ as the family of subsets of $V$ such that $S\in\C_{\P}$ if and only if $I_{\P}(S)=S$. Then it is easy to see that $(V,\C_{\P})$ is a finite convexity space, whose convex sets are precisely the fixed points of $I_{\P}$.

\begin{proposition}\label{prop:convexity}{\em \cite{V93}}
$(V,\C_{\P})$ is a finite convexity space.
\end{proposition}

If $\P$ is the collection of all the paths of $G$, then $(V,\C_{\P})$ is the {\em all-path convexity} with respect to $G$. In order to ease the notation, we omit the subscript ${\P}$ if it is clear from the context.

In this work we focus on seven important problems that are usually studied in the field of convexity in graphs, and solve them for the case of the all-path convexity. We need some additional definitions. Let $S\subseteq V$. If $I(S) = V$ then $S$ is an {\em interval set}. The {\em convex hull} $H(S)$ of $S$ is the smallest convex set containing $S$. Write $I^0(S)=S$ and define $I^{i+1}(S)=I(I^i(S))$ for $i\geq 0$. Note that $I(S)=I^1(S)$ and there is an index $i$ for which $H(S)=I^i(S)$. If $H(S) = V$ then $S$ is a {\em hull set}. The {\em convexity number} $c(G)$ of $G$ is the size of a maximum convex set $S\neq V$. The {\em interval number} $i(G)$ of $G$ is the size of a smallest interval set of $G$. The {\em hull number} $h(G)$ of $G$ is the size of a smallest hull set of $G$. Clearly, $h(G)\leq i(G)$. The {\em geodetic iteration number} of a graph $G$, denoted by $gin(G)$, is the smallest $i$ such that $H(S)=I^i(S)$ for every $S\subseteq V(G)$. Equivalently, $gin(G)=\min\{ i\mid I^i(S)=I^{i+1}(S)$ for every $S\subseteq V(G)\}$. Graphs with geodetic iteration number at most $1$ are called {\em interval monotone graphs}

Now we are in position to state the problems dealt with in this work:

\begin{itemize}
\setlength\itemsep{0em}
\item[--] {\sc Convex Set}\\
Input: A graph $G$ and a set $S\subseteq V$.\\
Question: Is $S$ convex?

\item[--]  {\sc Interval Determination}\\
Input: A graph $G$ and a set $S\subseteq V$.\\
Goal: Determine $I(S)$.

\item[--]  {\sc Convex Hull Determination}\\
Input: A graph $G$.\\
Goal: Determine $H(S)$.



\item[--]  {\sc Convexity Number}\\
Input: A graph $G$.\\
Goal: Determine $c(G)$.

\item[--]  {\sc Interval Number}\\
Input: A graph $G$.\\
Goal: Determine $i(G)$.

\item[--]  {\sc Hull Number}\\
Input: A graph $G$.\\
Goal: Determine $h(G)$.

\item[--]  {\sc Geodetic Iteration Number}\\
Input: A graph $G$.\\
Goal: Determine $gin(G)$.

\end{itemize}

\section{Main results}

In this section we present the solution of the problems stated in the preceding section for the case of the all-path convexity. For $X,Y\subseteq V$, denote $N(X,Y)=\{v\in Y\mid v \ \mbox{is a neighbor of some} \ x\in X\}$.

\begin{theorem}\label{thm:convex}
Let $S\subseteq V$. Then $S$ is convex in the all-path convexity if and only if either $S=V$ or for every connected component $G_i$ of $G-S$ it holds that $$|N(V(G_i),S)|=1.$$
\end{theorem}

\begin{proof}
Suppose that $S$ is convex and $S\neq V$, and let $G_i$ be a connected component of $G-S$. If $|N(V(G_i),S)|\geq 2$ then there are distinct $v,w\in S$ with neighbors $v',w'\in V(G_i)$ (not necessarily distinct). Let $P'=v'\ldots w'$ be a path from $v'$ to $w'$ in $G_i$ (such a path exists because $G_i$ is connected). Occasionally, $P'$ consists of a single vertex $v'=w'$. Now, let $P$ be the path $vP'w=vv'\ldots w'w$. This means that $v'$ (or $w'$) lies in a path with endpoints in $S$, that is, $v'\in I(S)$. Since $v'\notin S$, this implies $I(S)\neq S$, a contradiction to the assumption that $S$ is convex.

Conversely, suppose that $S\neq V$ and $|N(V(G_i),S)|=1$ for every connected component of $G-S$. Assume by contradiction that $S$ is not convex. Then, from the definition of convex set, there is a path $P=v_1v_2\ldots v_k$ with (distinct) endpoints in $S$ that contains vertices outside $S$. Let $j$ be the minimum index for which $v_j\in S$ and $v_{j+1}\not\in S$, and assume that $v_{j+1}\in V(G_i)$. Since the $G_i$'s are mutually isolated (i.e., there is no edge in $E$ joining a vertex in $G_i$ to a vertex in $G_k$ for $i\neq k$), let $l$ be the minimum index for which $l>j$ and $v_l\in S$ (in other words, after visiting $G_i$, the path $P$ must return to $S$ before visiting another component $G_k$). Thus $v_j,v_l\in N(V(G_i),S)$. But $j\neq l$, because $P$ is a path. Then $|N(V(G_i),S)|\geq 2$, a contradiction. Therefore, $S$ is convex.
\end{proof}

\begin{corollary}
{\sc Convex Set} can be solved in $O(n+m)$ time for the all-path convexity.
\end{corollary}

\begin{proof}
If $S\neq V$, computing the connected components of $G-S$ can be easily done in $O(n+m)$ time. In addition, checking whether $|N(V(G_i),S)|=1$ for every connected component $G_i$ of $G-S$ can also be done in $O(n+m)$ time. Thus, the corollary follows.
\end{proof}

Now we deal with the {\sc Interval Determination} problem. Consider the block decomposition of $G$ represented by a block-cut tree $T_G$, where each vertex of $T_G$ is associated with either a block $B_j$ (a cut edge or a maximal 2-connected subgraph of $G$) or a cut vertex $z_i\in V$, and there is an edge linking a vertex $B_{j}$ and a vertex $z_i$ in $T_G$ whenever block $B_{j}$ contains the cut vertex $z_i\in V$. This definition implies that the vertices of $T_G$ associated with blocks of $G$ form an independent set, and the same occurs for the vertices of $T_G$ associated with cut vertices of $G$. In addition, every leaf of $T_G$ represent a block of $G$. An {\em end block} of $G$ is a block associated with a leaf of $T_G$.

For a set $S\subseteq V$, let $T_S$ be the maximal subtree of $T_G$ such that every leaf of $T_S$ is associated with a block of $G$ containing a vertex of $S$ that is not a cut vertex in the subgraph $G_S$ induced by $\cup_{B_j\in V(T_S)} B_j$. Figure 1 shows an example.

\begin{figure}[!ht]
\includegraphics[width=\textwidth]{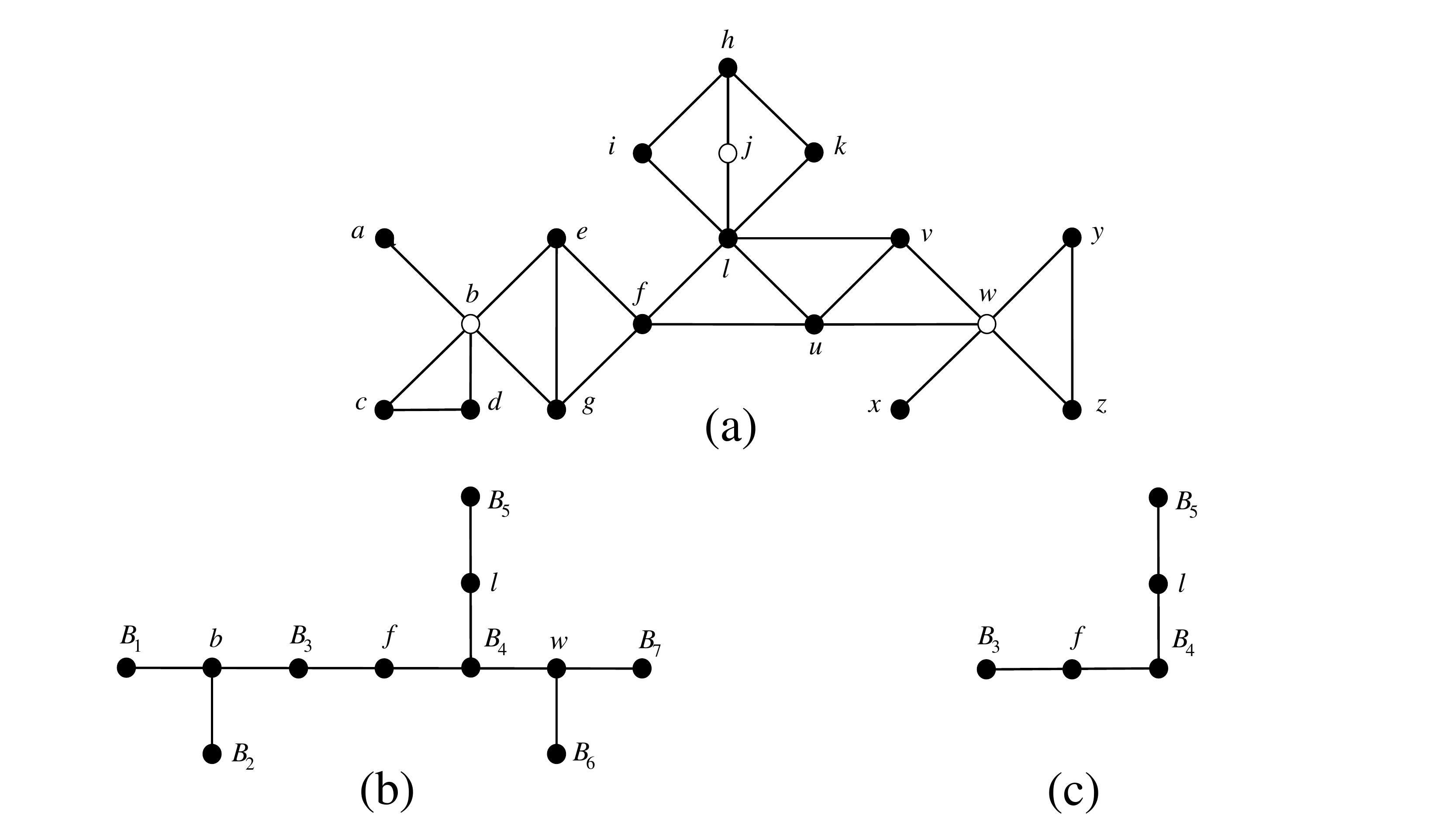}
\caption{(a) A graph $G$ and a subset $S=\{b,j,w\}$ (represented by the white vertices), whose blocks are such that $V(B_1)=\{a,b\}$, $V(B_2)=\{b,c,d\}$, $V(B_3)=\{b,e,g,f\}$, $V(B_4)=\{f,l,u,v,w\}$, $V(B_5)=\{h,i,j,k,l\}$, $V(B_6)=\{w,x\}$, $V(B_7)=\{w,y,z\}$; (b) block-cut tree $T_G$; (c) subtree $T_S$ of $T_G$. Block $B_3$ is a leaf of $T_S$ because it contains no cut vertex in the graph $G_S$ induced by $V(B_3)\cup V(B_4)\cup V(B_5)$.}
\end{figure}

\begin{lemma}\label{lem:path}
Let $S\subseteq V$, and let $u,w$ be two distinct vertices in $S$, belonging to blocks $B_u$ and $B_w$ of $T_S$, respectively. Assume that $u$ and $w$ are not cut vertices in $G_S$. Let $B_{j_1}z_1B_{j_2}z_2\ldots z_{k-1}B_{j_k}$ be a path in $T_S$ between $B_{j_1}=B_u$ and $B_{j_k}=B_w$. Then, for every $v\in\cup_{i=1}^{k} V(B_{j_i})$, there is a path $P$ in $G$ from $u$ to $w$ passing through $v$.
\end{lemma}

\begin{proof}
If $B_u=B_w$ and there is a vertex $v\in V(B_u)\setminus\{u,w\}$, by the Fan Lemma (Proposition 9.5 in~\cite{BM08}, applied to $2$-connected graphs) there is a path $P$ from $u$ to $w$ passing through $v$.

If $B_u\neq B_w$, construct from $B_{j_1}z_1B_{j_2}z_2\ldots z_{k-1}B_{j_k}$ a path $P$ in $G$ as follows. Let $v\in\cup_{i=1}^{k} V(B_{j_i})$ and suppose first that $v$ is not a cut vertex in $G_S$. Assume that $B_v=B_{j_s}$ for some $1<s<k$ (the cases $s=1$ or $s=k$ are similar).
\begin{itemize}
\setlength\itemsep{0em}
\item[--] Let $P_1$ be a path contained in $B_{j_1}$ starting at $u$ and ending at $z_1$.
\item[--] Let $P_s$ be a path contained in $B_{j_s}$ starting at $z_{s-1}$, passing through $v$, and ending at $z_s$. Since $v$ is not a cut vertex of $G_S$, the block $B_{j_s}$ contains at least three vertices. Thus, the Fan Lemma guarantees the existence of path $P_s$.
\item[--] Let $P_i$, $2\leq i\leq k-1$, $i\neq s$, be a path contained in $B_{j_i}$ starting at $z_{i-1}$ and ending at $z_i$.
\item[--] Finally, let $P_k$ be a path contained in $B_{j_k}$ starting at $z_{k-1}$ and ending at $w$.
\end{itemize}
Now, define $P$ as the concatenation $P_1P_2\ldots P_k$ (remove replicated vertices that appear consecutively). Then $P$ is a path from $u$ to $w$ passing through vertex $v$.

The case where $v$ is a cut vertex of $G_S$ is dealt with similarly. In this case, there is no need of applying the Fan Lemma to $B_{j_s}$. $P_s$ can be simply taken as a path from $z_{s-1}$ to $z_s$, since $v\in\{z_{s-1},z_s\}$.
\end{proof}

Let $S\subseteq V$ and $v\in V\setminus S$. Say that a vertex $z\neq v$ {\em separates} $v$ {\em and} $V'\subseteq V\setminus\{v\}$ if, in graph $G-z$, the connected component containing $v$ contains no vertex of $V'$.

\begin{theorem}\label{thm:i-s}
Let $S\subseteq V$ with $|S|\geq 2$. Then
$$I(S)= \bigcup_{B_j\in V(T_S)} B_j.$$
\end{theorem}

\begin{proof}
Let $v\in \cup_{B_j\in V(T_S)} B_j$. If $v\in S$ then $v\in I(S)$. If $v\not\in S$, consider a block $B_v$ containing $v$. Clearly, there is a maximal path $P_v=B_{j_1}z_1B_{j_2}z_2\ldots z_{k-1}B_{j_k}$ in $T_S$ containing $B_v$. Thus, $v\in\cup_{i=1}^{k} V(B_{j_i})$. By the maximality of $P_v$, $B_{j_1}$ and $B_{j_k}$ are leaves in $T_S$, and this implies that $B_{j_1}$ contains a vertex $u\in S$ and $B_{j_k}$ contains a vertex $w\in S$, $w\neq u$, such that $u$ and $w$ are not cut vertices in $G_S$. Then, by Lemma~\ref{lem:path}, there is a path $P$ in $G$ from $u$ to $w$ passing through $v$. Therefore, $v\in I(S)$, i.e., $\cup_{B_j\in V(T_S)} B_j \ \subseteq \  I(S)$.

Now, let $v\in I(S)$. If $v\in S$, then, from the definition of $T_S$, it is clear that $v\in\cup_{B_j\in V(T_S)} B_j$. Assume then $v\not\in S$ and $v\not\in\cup_{B_j\in V(T_S)} B_j$. Thus, $v$ is in a block $B_v$ of $T_G$ that is not in $T_S$, and this implies that there is a cut vertex $z\in V(B_v)$, $z\neq v$, that separates $v$ and $\cup_{B_j\in V(T_S)} B_j$. Since $S\subseteq\cup_{B_j\in V(T_S)} B_j$, $z$ separates $v$ and $S$, and therefore every path from some $u\in S$ to $v$ passes through $z$, i.e., it is impossible to find a path from a vertex $u\in S$ to another vertex $w\in S$ passing through $v$. This contradicts the assumption $v\in I(S)$. Hence, $v\in\cup_{B_j\in V(T_S)} B_j$ and $I(S)\subseteq\cup_{B_j\in V(T_S)} B_j$.
\end{proof}

\if 10


\begin{theorem}
Let $S\subseteq V$ and $v\in V\setminus S$. Then $v\in I(S)$ if and only if there are distinct $u,w\in S$ for which no vertex $z\neq v$ separates $v$ and $\{u,w\}$.
\end{theorem}

\begin{proof}
Suppose first that $v\in I(S)$. Then, from the definition of the interval function $I$, there is a path $P$ with distinct endpoints $u,w\in S$ such that $v$ lies in $P$. Thus, the removal of any $z\not\in V(P)$ keeps $u,v,w$ in the same connected component of $G-z$ (the one that contains the path $P$). On the other hand, the removal of any $z\in V(P)$ keeps $u,v$ or $w,v$ in the same connected component of $G-z$. In any case, no vertex $z\neq v$ separates $v$ and $\{u,w\}$.

Conversely, suppose that there are distinct $u,w\in S$ for which no vertex $z\neq v$ separates $v$ and $\{u,w\}$.

\begin{itemize}
\setlength\itemsep{0em}
\item[--] If $u$, $v$, and $w$ belong to a same block $B_1$ then there is a path $P$ contained in $B_1$ from $u$ to $w$ passing through $v$ \ \ -- this is precisely the Fan Lemma (Proposition 9.5 in \cite{BM08}) applied to the case of $2$-connected graphs. Thus $v\in I(S)$.
\item[--] If $u$ and $w$ belong to a block $B_1$ and $v$ belongs to a block $B_v\neq B_1$ then the edges adjacent to $B_1$ in $T$ can be assumed to have labels distinct from $v$ (otherwise $v$ is a vertex of $B_1$ and we would fall back to the previous case). This implies that there is an edge in $T$ with a label $z\neq v$ whose removal disconnects $B_1$ from $B_v$, that is, there is a vertex $z\neq v$ in $V$ that separates $v$ and $\{u,w\}$. This is a contradiction, and thus this case cannot occur.
\item[--] If $u$ and $v$ belong to a block $B_1$ and $w$ belongs to a block $B_w\neq B_1$ (the case $w$ and $v$ in $B_1$ and $u$ in $B_u\neq B_1$ is similar) then the edges adjacent to $B_1$ in $T$ can be assumed to have labels distinct from $w$ (otherwise we would fall back to the first case again). In this situation we take the only path $P_T$ in $T$ between $B_1$ and $B_w$, say $P_T=B_1B_2\ldots B_k$ (where $B_k=B_w$) and construct from $P_T$ a path $P$ in $G$ as follows. Denote by $z_i$ the label of the edge $B_iB_{i+1}$, $1\leq i\leq k-1$. Let $P_1$ be a path contained in $B_1$ starting at $u$, passing through $v$ and ending at $z_1$ (if $v\neq z_1$, the Fan Lemma guarantees the existence of such a path). Let $P_i$, $2\leq i\leq k-1$, be a path contained in $B_i$ starting at $z_i$ and ending at $z_{i+1}$. Finally, let $P_k$ be a path contained in $B_k$ starting at $z_k$ and ending at $w$. Now, define $P$ as the concatenation $P_1P_2\ldots P_k$ (remove replicated vertices that appear consecutively). Since $P$ is a path from $u$ to $w$ passing through $v$, we conclude that $v\in I(S)$.
\item [--] Finally, if there are three distinct blocks $B_u$, $B_v$, and $B_w$ containing, respectively, $u$, $v$, and $w$, then there must be a path $P_T$ in $T$ from $B_u$ to $B_w$ passing through $B_v$, for otherwise there would be an edge $z\in E(T)$, $z\neq v$, whose removal disconnects $B_v$ from $B_u$ and $B_w$ (which implies the existence of a vertex $z\neq v$ that separates $v$ and $\{u,w\}$ in $G$). From the path $P_T$ we can construct a path $P$ in $G$ from $u$ to $w$ passing through $v$, using the same strategy as explained in the previous case. Thus $v\in I(S)$.
\end{itemize}
\vspace{-0.7cm}
\end{proof}

\begin{corollary}
{\sc Interval Determination} can be solved in $O(n(n+m))$ time for the all-path convexity.
\end{corollary}

\begin{proof}
To solve the problem, we can run the following algorithm. First, for all $z\neq v$, do: determine the connected components of $G-z$, locate the connected component $C_z$ of $G-z$ containing $v$, and compute $S_z = S\setminus C_z$. (Note that $z$ separates $v$ and $S_z$).

Having computed all the subsets $S_z$, for $z\neq v$, compute $S'=S\setminus (\cup_{z\neq v} S_z)$. If $S'$ contains at least two vertices $u$ and $w$ then the answer to the problem is ``yes'', since no $z\neq v$ separates $v$ and $\{u,w\}$; otherwise the answer is ``no''.

It is easy to see that computing $S_z$, for some $z\neq v$, takes $O(n+m)$ time. Thus the entire algorithm runs in $O(n(n+m))$ time.
\end{proof}


\fi

\begin{corollary}\label{coro:id}
{\sc Interval Determination} can be solved in $O(n+m)$ time for the all-path convexity.
\end{corollary}

\begin{proof}
Assume $|S|\geq 2$ (other cases are trivial). The block-cut tree $T_G$ of $G$ can be constructed in $O(n+m)$ time using well-known algorithmic techniques~\cite{HT73}. The determination of $T_S$ can be easily done in $O(n+m)$ time as follows: initialize $T_S\leftarrow T_G$ and repeat the following procedure until no longer possible: if $B_j$ is a leaf of the tree containing no vertices of $S$ or containing exactly one vertex of $S$ that is a cut vertex of $G_S$ then remove from the tree the vertex associated with $B_j$ and the cut vertex $z_i$ adjacent to $B_j$ in the tree. Each removal takes $O(|V(B_j)|)$ time, thus the construction of $T_S$ from $T_G$ takes $O(n)$ time. Finally, computing $\bigcup_{B_j\in V(T_S)} B_j$ takes $O(n)$ time. Hence, the theorem follows.
\end{proof}

In Figure 1, for $S=\{b,j,w\}$, we have that $I(S)=V(B_3)\cup V(B_4)\cup V(B_5)$.

\begin{lemma}\label{lem:i-s-convex}
Let $S\subseteq V$. Then, in the all-path convexity, $I(S)$ is always a convex set.
\end{lemma}
\begin{proof}
Assume $S\neq V$. By Theorem~\ref{thm:i-s}, $I(S)= \cup_{B_j\in V(T_S)} B_j$. Let $X= \cup_{B_j\notin V(T_S)} B_j$. In other words, $X$ is formed by all vertices belonging to blocks of $G$ that are {\em not} represented in $T_S$. It is clear that every vertex in $I(S)\cap X$ is a cut vertex of $G$. That is, the condition
\begin{center}
``$|N(V(G_i),I(S))|=1$ for every connected component $G_i$ of $G-I(S)$''
\end{center}
is valid. Hence, by Theorem~\ref{thm:convex}, $I(S)$ is a convex set.
\end{proof}

\begin{theorem}\label{thm:h-s}
In the all-path convexity, $H(S)=I(S)$ for every $S\subseteq V$.
\end{theorem}
\begin{proof}
Follows from Lemma~\ref{lem:i-s-convex}.
\end{proof}

\begin{corollary}\label{coro:gin}
In the all-path convexity, every graph $G$ satisfies $gin(G)\leq 1$ (i.e., is interval monotone).
\end{corollary}
\begin{proof}
Follows from the previous theorem.
\end{proof}

\begin{corollary}
{\sc Convex Hull Determination} can be solved in $O(n+m)$ time for the all-path convexity.
\end{corollary}
\begin{proof}
Follows from Corollary~\ref{coro:id} and Theorem~\ref{thm:h-s}.
\end{proof}

\begin{corollary}
{\sc Geodetic Iteration Number} can be solved in $O(1)$ time for the all-path convexity.
\end{corollary}
\begin{proof}
Follows from Corollary~\ref{coro:gin}.
\end{proof}

Now we focus on the {\sc Convexity Number} problem. For an end block $B_j$ of $G$, let $|V(B_j)|=b_j$. Define $b(G)=\min\{b_j\mid B_j \ \mbox{is an end block of} \ G\}$.

\begin{theorem}\label{thm:c-g}
In the all path convexity, for any graph $G$ it holds that:

\[ c(G) = \left\{ \begin{array}{ll}
         1, & \mbox{if $|V|=2$ or $G$ is $2$-connected};\\
         n-b(G)+1, & \mbox{otherwise}.\end{array} \right. \]

\end{theorem}

\begin{proof}
If $|V|=2$ then the theorem is trivially true. If $G$ is $2$-connected then, by the Fan Lemma, for each pair of vertices $u,w\in V, w\neq u$, we have that every $v\notin\{u,w\}$ lies in a path from $u$ to $w$. Hence, for every $S$ with $2\leq |S|\leq n-1$, $S$ is not convex. This implies that $c(G)=1$.

Suppose now that $G$ is not $2$-connected. Note that any $S\subseteq V$ consisting of the union of all vertex sets of all blocks of $G$ but one end block, say $B_j$, is convex, because the only cut vertex $z$ belonging to $V(B_j)$ separates all the vertices of $V\setminus V(B_j)$ from $V(B_j)\setminus z$. (Note that $z\in S$.) Thus the maximum convex set in $G$ is obtained by removing from $G$ all the vertices in an end block $B_j$ with minimum size, except the cut vertex $z\in V(B_j)$.
\end{proof}

\begin{corollary}
{\sc Convexity Number} can be solved in $O(n+m)$ time for the all-path convexity.
\end{corollary}

\begin{proof}
Follows from Theorem~\ref{thm:c-g} and the fact that the block decomposition of $G$ can be obtained in $O(n+m)$ time.
\end{proof}

Now we finally consider the problems {\sc Interval Number} and {\sc Hull Number}. Let $eb(G)$ be the number of end blocks of $G$.

\begin{theorem}\label{thm:i-g}
In the all path convexity, for any graph $G$ it holds that:

\[ i(G) = \left\{ \begin{array}{ll}
         1, & \mbox{if $G$ is trivial};\\
         2, & \mbox{if $|V|=2$ or $G$ is $2$-connected};\\
         eb(G), & \mbox{otherwise}.\end{array} \right. \]

\end{theorem}

\begin{proof}
If $|V|=1$ or $|V|=2$ then the theorem is trivially true. If $G$ is $2$-connected, by the Fan Lemma any pair $u,w\in V$, $w\neq u$, is an interval set, and thus $i(G)=2$ in this case.

Finally, if $G$ is not $2$-connected, consider $S\subseteq V$ such that $S\cap V(B_j)=\{v_j\}$ for each end block $B_j$ of $G$, where $v_j$ is not a cut vertex of $G$. Note that $|S|=eb(G)$. The definition of $S$ implies that $T_S=T_G$, and thus every vertex $v\in V$ is in a block $B_v$ of $G$ belonging to a maximal path $B_{j_1}z_1B_{j_2}z_2\ldots z_{k-1}B_{j_k}$ in $T_S$ such that $B_{j_1}$ and $B_{j_k}$ are end blocks of $G_S=G$ containing, respectively, vertices $u,w\in S$, $w\neq u$, that are not cut vertices in $G_S=G$. By Lemma~\ref{lem:path}, there is a path $P$ in $G$ from $u$ to $w$ passing through $v$. In other words, $S$ is an interval set. To conclude the proof, if a set $S'\subseteq V$ is such that $|S'|<eb(G)$ then there is at least one end block $B_j$ in $G$ such that $V(B_j)\setminus\{z_j\}$ contains no vertices of $S'$, where $z_j$ is the cut vertex of $G$ belonging to $V(B_j)$. Hence, no vertex in $V(B_j)\setminus\{z_j\}$ can lie in a path starting and ending at distinct vertices of $S'$, i.e., $S'$ cannot be an interval set. Thus $S$ is minimum.
\end{proof}

\begin{corollary}\label{coro:i=h}
In the all path convexity, for any graph $G$ it holds that $h(G)=i(G)$.
\end{corollary}

\begin{proof}
Follows from Theorem~\ref{thm:h-s}.
\end{proof}

\begin{corollary}
{\sc Interval Number} and {\sc Hull Number} can be solved in $O(n+m)$ time for the all-path convexity.
\end{corollary}

\section{Concluding remarks}

In this work we showed that the problems {\sc Convex Set}, {\sc Interval Determination}, {\sc Convex Hull Determination}, {\sc Convexity Number}, {\sc Interval Number}, and {\sc Hull Number} can all be solved in $O(n+m)$ time for the all-path convexity, whereas {\sc Geodetic Iteration Number} can be solved in $O(1)$ time (all graphs are interval monotone with respect to the all-path convexity). The arguments used to solve such problems are based on the block decomposition and the block-cut tree of the graph.


\end{document}